\documentclass[oneside,english]{amsart}
\usepackage[T1]{fontenc}
\usepackage[latin9]{inputenc}
\usepackage{prettyref}
\usepackage{bm}
\usepackage{amstext}
\usepackage{amsthm}
\usepackage{amssymb}
\usepackage[all]{xy}

\makeatletter
\numberwithin{equation}{section}
\numberwithin{figure}{section}
\theoremstyle{plain}
\newtheorem{thm}{\protect\theoremname}[section]
\theoremstyle{definition}
\newtheorem{example}[thm]{\protect\examplename}
\theoremstyle{remark}
\newtheorem{rem}[thm]{\protect\remarkname}
\theoremstyle{definition}
\newtheorem{defn}[thm]{\protect\definitionname}
\theoremstyle{plain}
\newtheorem{lem}[thm]{\protect\lemmaname}
\theoremstyle{plain}
\newtheorem{prop}[thm]{\protect\propositionname}
\theoremstyle{plain}
\newtheorem{cor}[thm]{\protect\corollaryname}

\usepackage{mathtools}
\usepackage{braket}
\usepackage{prettyref}
\usepackage{tikz}
\usetikzlibrary{calc}

\newrefformat{prop}{Proposition \ref{#1}}
\newrefformat{cor}{Corollary \ref{#1}}
\newrefformat{rem}{Remark \ref{#1}}
\newrefformat{subsec}{Subsection \ref{#1}}

\newcommand{\ns}[1]{\prescript{\ast}{}{#1}}
\newcommand{\Top}{\mathbf{Top}}
\newcommand{\CR}{\mathbf{CR}}
\newcommand{\Unif}{\mathbf{Unif}}
\newcommand{\uE}{\mathbf{uE}}

\@ifundefined{showcaptionsetup}{}{%
 \PassOptionsToPackage{caption=false}{subfig}}
\usepackage{subfig}
\makeatother

\usepackage{babel}
\providecommand{\corollaryname}{Corollary}
\providecommand{\definitionname}{Definition}
\providecommand{\examplename}{Example}
\providecommand{\lemmaname}{Lemma}
\providecommand{\propositionname}{Proposition}
\providecommand{\remarkname}{Remark}
\providecommand{\theoremname}{Theorem}

\begin{document}
\title[Vietoris-type and microsimplicial homology theories]{Relationship among various Vietoris-type and microsimplicial homology
theories}
\author{Takuma Imamura}
\address{Research Institute for Mathematical Sciences\\
Kyoto University\\
Kitashirakawa Oiwake-cho, Sakyo-ku, Kyoto 606-8502, Japan}
\email{timamura@kurims.kyoto-u.ac.jp}
\begin{abstract}
In this paper, we clarify the relationship among the Vietoris-type
homology theories and the microsimplicial homology theories, where
the latter are nonstandard homology theories defined by M. C. McCord
(for topological spaces), T. Korppi (for completely regular topological
spaces) and the author (for uniform spaces). We show that McCord's
and our homology are isomorphic for all compact uniform spaces and
that Korppi's and our homology are isomorphic for all fine uniform
spaces. Our homology shares many good properties with Korppi's homology.
As an example, we outline a proof of the continuity of our homology
with respect to uniform resolutions. S. Garavaglia proved that McCord's
homology is isomorphic to Vietoris homology for all compact topological
spaces. Inspired by this result, we prove that our homology is isomorphic
to uniform Vietoris homology for all precompact uniform spaces and
that Korppi's homology is isomorphic to normal Vietoris homology for
all pseudocompact completely regular topological spaces.
\end{abstract}

\maketitle

\section{Introduction}

Nonstandard homology theories have been developed for various spaces
in the existing literature. M. C. McCord \cite{McC72} introduced
a nonstandard homology of topological spaces, the first investigation
of nonstandard homology. By imitating McCord's definition, T. Korppi
\cite{Kor07,Kor12} constructed another nonstandard homology for completely
regular spaces, and the author \cite{Ima16} defined a similar one
for uniform spaces, called $\mu$-homology. These homology theories
are all \emph{microsimplicial}, i.e., based on hyperfinite chains
of infinitesimally small (micro) simplices. Other types of nonstandard
homology can be found in \cite[Section 7]{Ima16}.

McCord mentioned a similarity between microsimplicial and Vietoris-type
homology theories. Typical examples of Vietoris-type are as follows:
Vietoris homology of topological spaces, normal Vietoris homology
of completely regular spaces, and uniform Vietoris homology of uniform
spaces. \emph{Informally}, Vietoris-type homology is the homology
of the Vietoris complex of some ``infinitely fine'' cover. Formally,
Vietoris-type homology is defined as the inverse limit of the homologies
of Vietoris complexes, where the limit runs over some directed set
of covers with respect to refinement. Note that Vietoris' construction
indirectly deals with infinitely fine covers, just like epsilon\textendash delta
analysis does with infinitesimals. There is another way to formulate
the informal concept ``infinitely fine'', namely, the use of nonstandard
analysis. Nonstandard analysis enables us directly to deal with infinitely
fine covers. Microsimplicial homology is \emph{formally} defined as
the homology of the Vietoris complex of some infinitely fine cover
in the nonstandard sense. Under these circumstances, it is natural
to expect that each microsimplicial homology is isomorphic to the
corresponding Vietoris-type homology. Indeed, Garavaglia \cite{Gar78}
proved that McCord homology is isomorphic to Vietoris homology for
all compact spaces. The equivalence, however, does not hold for some
non-compact spaces. This means that McCord homology is \emph{not}
just a paraphrase of Vietoris homology. Vietoris-type homology theories
are inexact even for compact metrisable pairs (see \cite[Example 2, p.126]{MS82}),
because they are defined by inverse limits which do not preserve exact
sequences. By contrast, microsimplicial homology theories are exact
for all pairs of spaces. This is an advantage of microsimplicial-type
compared with Vietoris-type.

This paper aims to clarify the relationship among Vietoris-type and
microsimplicial homology theories. \prettyref{fig:Rel-between-V/M-homs}
illustrates the relationship among the homology theories mentioned
so far.
\begin{figure}
\[
\xymatrix{ & \text{Vietoris-type} &  & \text{microsimplicial}\\
\text{topological} & \text{Vietoris}\ar@{=}[d]_{\text{paracompact}}\ar@{=}[rr]^{\text{compact}} &  & \text{McCord}\ar@{=}[d]^{\text{compact}}\\
\text{completely regular} & \text{normal Vietoris}\ar@{=}[d]_{\text{fine}}\ar@{=}[rr]^{\text{pseudocompact}} &  & \text{Korppi}\ar@{=}[d]^{\text{fine}}\\
\text{uniform} & \text{uniform Vietoris}\ar@{=}[rr]^{\text{precompact}} &  & \mu
}
\]
\caption{\label{fig:Rel-between-V/M-homs}Double lines indicate isomorphisms.
The label of each line indicates the condition for that isomorphism.}
\end{figure}

In \prettyref{sec:Rel-to-Korppi-homology}, we deal with the microsimplicial
homology theories. We show that
\begin{itemize}
\item McCord homology and $\mu$-homology are exactly the same for all compact
uniform spaces;
\item Korppi homology and $\mu$-homology are exactly the same for all fine
uniform spaces.
\end{itemize}
We remark that these equalities do not hold in general. Because of
the second equality, we can regard $\mu$-homology as a generalisation
of Korppi homology from fine uniform spaces to general uniform spaces.
$\mu$-homology inherits many properties from Korppi homology. For
example, uniform Vietoris homology with standard coefficients can
be embedded into $\mu$-homology with nonstandard coefficients. This
is a generalisation of the fact that normal Vietoris homology can
be embedded into Korppi homology. As another example, we outline a
proof of the continuity of $\mu$-homology with respect to uniform
resolutions.

In \prettyref{sec:Rel-between-uniform-Vietoris-and-mu-homology},
we prove the following two isomorphisms (inspired by Garavaglia's
result):
\begin{itemize}
\item $\mu$-homology is isomorphic to uniform Vietoris homology for all
precompact uniform spaces;
\item Korppi homology is isomorphic to normal Vietoris homology for all
pseudocompact completely regular spaces.
\end{itemize}
To prove the former isomorphism, we introduce the notion of S-homotopy,
the singular homotopy on the category of uniform spaces with S-continuous
maps (which admit infinitesimal discontinuity), and show that uniform
Vietoris homology satisfies the S-homotopy axiom for all precompact
uniform spaces. The latter isomorphism immediately follows from the
former one. We also prove that these isomorphisms do not hold in general.

$\mu$-homology can be extended to the category of nonstandard subsets
of uniform spaces with nonstandardly continuous maps. We briefly discuss
this extension in \prettyref{sec:Homology-of-nonstandard-subsets}.

Finally, in \prettyref{sec:Use-of-the-saturation-principle}, we remark
the necessity and unnecessity of the saturation principle. This remark
will be important for considering the microsimplicial homology theories
in general nonstandard models.

\section{\label{sec:Preliminaries}Preliminaries}

We use the basic terminology in uniform topology (see the monograph
\cite{Isb64} by Isbell). We assume the reader to be familiar with
nonstandard mathematics, in particular, nonstandard topology. For
basic notions and results in nonstandard topology, we refer to Robinson
\cite{Rob66} and Stroyan and Luxemburg \cite{SL76}.

\subsection{Basic settings of nonstandard analysis}

As in \cite{Ima16}, we use Robinson-style nonstandard analysis. We
fix a transitive universe $\mathbb{U}$, called the standard universe,
that satisfies sufficiently many (but only finitely many) axioms of
ZFC. All standard objects we consider belong to $\mathbb{U}$. More
specifically, using the reflection principle, we pick a large enough
ordinal $\lambda$ and let $\mathbb{U}:=V_{\lambda}$, where $V_{\lambda}$
is the cumulative hierarchy of rank $\lambda$. We also fix an elementary
extension $\ns{\mathbb{U}}$ of $\mathbb{U}$, called the internal
universe, that is $\left|\mathbb{U}\right|^{+}$-saturated. In particular,
$\ns{\mathbb{U}}$ is an enlargement of $\mathbb{U}$. By saying the
words `transfer', `saturation' and `weak saturation', we indicate
the elementary equivalence between $\mathbb{U}$ and $\ns{\mathbb{U}}$,
the saturation property of $\ns{\mathbb{U}}$ and the enlargement
property of $\ns{\mathbb{U}}$ relative to $\mathbb{U}$, respectively.
The map $x\mapsto\ns{x}$ denotes the elementary embedding $\mathbb{U}\hookrightarrow\ns{\mathbb{U}}$.
We omit the star of $\ns{x}$ when $x$ is considered to be an atomic
object (such as a number and a point). Given a concept $X$ on $\mathbb{U}$
definable by a first-order formula with parameters in $\mathbb{U}$,
the associated concept on $\ns{\mathbb{U}}$ definable by the same
formula is called\emph{ internal $X$}, \emph{hyper $X$}, $\ns{X}$,
etc.

\subsection{Vietoris-type homology theories}

We here recall the definition schema of Vietoris-type homology theories
\cite{Dow52}.

Let $X$ and $Y$ be sets and let $R$ be a subset of $X\times Y$.
The \emph{Vietoris complex} of $\left(X,Y,R\right)$ is the simplicial
set $\mathcal{V}\left(X,Y,R\right)$ whose vertices are the points
of $X$ and whose vertices $a_{0},\ldots,a_{p}$ span a simplex if
and only if there exists a $b\in Y$ such that $a_{i}\mathrel{R}b\ \left(0\leq i\leq p\right)$.
\prettyref{fig:ex-of-V-complex} gives an example of a Vietoris complex.
\begin{figure}
\subfloat[]{\begin{tikzpicture}
\node (a) at (1, {sqrt(3)}) {$a$};
\node (b) at (0, 0) {$b$};
\node (c) at (2, 0) {$c$};
\draw (0.5, {sqrt(3) / 2}) circle [x radius=1.5, y radius=0.5, rotate=60];
\draw (1, 0) circle [x radius=1.5, y radius=0.5];
\draw (1.5, {sqrt(3) / 2}) circle [x radius=1.5, y radius=0.5, rotate=120];
\end{tikzpicture}}\subfloat[]{\begin{tikzpicture}
\fill (1, {sqrt(3)}) circle[radius=1pt] node [above=2pt] {$a$};
\fill (0, 0) circle[radius=1pt] node [left=2pt] {$b$};
\fill (2, 0) circle[radius=1pt] node [right=2pt] {$c$};
\draw (1, {sqrt(3)})--(0, 0)--(2, 0)--cycle;
\end{tikzpicture}}\caption{\label{fig:ex-of-V-complex}The left depicts the three point space
$X:=\set{a,b,c}$ equipped with the cover $\lambda:=\set{\set{a,b},\set{a,c},\set{b,c}}$.
The right depicts a geometric realisation of $\mathcal{V}\left(X,\lambda,\in\right)$
within the plane.}
\end{figure}

Let $X$ be a set and let $\lambda$ and $\mu$ be covers of $X$.
$\lambda$ is called a \emph{refinement} of $\mu$ if there is a map
$\varphi\colon\lambda\to\mu$ such that $U\subseteq\varphi\left(U\right)$
for all $U\in\lambda$. The cover $\lambda\wedge\mu:=\set{U\cap V|U\in\lambda,V\in\mu}$
is a common refinement of $\lambda$ and $\mu$. Hence, the set of
all covers forms a (downward) directed set under refinement.

Let $G$ be an abelian group. Let $\left(X,A\right)$ be a pair of
sets (such that $A\subseteq X$) and $\mathcal{D}_{X}$ a directed
set of covers of $X$. Given $\lambda\in\mathcal{D}_{X}$, we denote
by $\left(X_{\lambda},A_{\lambda}\right)$ the simplicial pair $\left(\mathcal{V}\left(X,\lambda,\in\right),\mathcal{V}\left(A,\lambda,\in\right)\right)$.
If $\lambda$ is a refinement of $\mu$, $X_{\lambda}$ and $A_{\lambda}$
are simplicial subsets of $X_{\mu}$ and $A_{\mu}$, respectively.
Let $i_{\mu\lambda}\colon\left(X_{\lambda},A_{\lambda}\right)\hookrightarrow\left(X_{\mu},A_{\mu}\right)$
be the inclusion and let $p_{\mu\lambda}:=H_{\bullet}\left(i_{\mu\lambda};G\right)$.
Thus we have an inverse system $\left(\mathcal{D}_{X},H_{\bullet}\left(X_{\lambda},A_{\lambda};G\right),p_{\mu\lambda}\right)$,
called the \emph{Vietoris system} for $\left(X,A\right)$. Here $H_{\bullet}\left(\cdot;G\right)$
denotes the ordinary homology functor of simplicial pairs with coefficients
in $G$. The \emph{Vietoris homology} of $\left(X,A\right)$ with
coefficients in $G$ with respect to $\mathcal{D}_{X}$ is the inverse
limit
\[
\check{H}_{\bullet}^{\mathcal{D}_{X}}\left(X,A;G\right):=\lim_{\lambda\in\mathcal{D}_{X}}H_{\bullet}\left(X_{\lambda},A_{\lambda};G\right).
\]

\begin{example}
Let $\Top_{2}$ be the category of topological pairs with continuous
maps. We denote by $\mathcal{O}_{X}$ the directed set of all open
covers of a topological space $X$. The \emph{Vietoris homology} of
a topological pair $\left(X,A\right)$ is $\check{H}_{\bullet}\left(X,A;G\right):=\check{H}_{\bullet}^{\mathcal{O}_{X}}\left(X,A;G\right)$.
One can extend $\check{H}_{\bullet}\left(\cdot;G\right)$ to a functor
from $\Top_{2}$ to the category of graded abelian groups as follows:
let $f\colon\left(X,A\right)\to\left(Y,B\right)$ be a continuous
map. Given $\lambda\in\mathcal{O}_{Y}$, its pullback $f^{-1}\lambda:=\set{f^{-1}\left(U\right)|U\in\lambda}$
is in $\mathcal{O}_{X}$. Hence, $f$ canonically induces a homomorphism
$f_{\lambda}\colon\check{H}_{\bullet}\left(X,A;G\right)\to H_{\bullet}\left(Y_{\lambda},B_{\lambda};G\right)$.
If $\lambda$ is a refinement of $\mu$, then the diagram
\[
\xymatrix{ & \check{H}_{\bullet}\left(X,A;G\right)\ar[dl]_{f_{\lambda}}\ar[dr]^{f_{\mu}}\\
H_{\bullet}\left(Y_{\lambda},B_{\lambda};G\right)\ar[rr]^{p_{\mu\lambda}} &  & H_{\bullet}\left(Y_{\mu},B_{\mu};G\right)
}
\]
commutes. By the universal property of $\check{H}_{\bullet}\left(Y,B;G\right)$,
there exists a unique homomorphism $\check{H}_{\bullet}\left(f;G\right)\colon\check{H}_{\bullet}\left(X,A;G\right)\to\check{H}_{\bullet}\left(Y,B;G\right)$
that makes the diagram
\[
\xymatrix{\check{H}_{\bullet}\left(X,A;G\right)\ar[dr]_{f_{\lambda}}\ar[rr]^{\check{H}_{\bullet}\left(f;G\right)} &  & \check{H}_{\bullet}\left(Y,B;G\right)\ar[dl]^{p_{\lambda}}\\
 & H_{\bullet}\left(Y_{\lambda},B_{\lambda};G\right)
}
\]
commutative for every $\lambda\in\mathcal{O}_{Y}$, where $p_{\lambda}$
is the canonical projection. It is easy to verify that $\check{H}_{\bullet}\left(\cdot;G\right)$
is functorial.
\end{example}

\begin{example}
Let $\CR_{2}$ be the category of completely regular pairs with continuous
maps. We denote by $\mathcal{N}_{X}$ the directed set of all normal
covers of a completely regular space $X$. The \emph{normal Vietoris
homology} of a completely regular pair $\left(X,A\right)$ is $\check{H}_{\bullet}^{n}\left(X,A;G\right):=\check{H}_{\bullet}^{\mathcal{N}_{X}}\left(X,A;G\right)$.
Like above, one can extend $\check{H}_{\bullet}\left(\cdot;G\right)$
to a functor on $\CR_{2}$.
\end{example}

\begin{example}
Let $\Unif_{2}$ be the category of uniform pairs with uniformly continuous
maps. We denote by $\mathcal{U}_{X}$ the directed set of all uniform
covers of a uniform space $X$. The \emph{uniform Vietoris homology}
of a uniform pair $\left(X,A\right)$ is $\check{H}_{\bullet}^{u}\left(X,A;G\right):=\check{H}_{\bullet}^{\mathcal{U}_{X}}\left(X,A;G\right)$.
One can extend $\check{H}_{\bullet}\left(\cdot;G\right)$ to a functor
on $\Unif_{2}$ as before.
\end{example}

\begin{rem}
Uniform Vietoris homology has another definition in terms of entourages
instead of uniform covers. Let $\left(X,A\right)$ be a uniform pair.
Let $\mathcal{E}_{X}$ be the set of all entourages of $X$. $\mathcal{E}_{X}$
forms a (downward) directed set under inclusion. Given $U\in\mathcal{E}_{X}$,
we denote by $\left(X_{U},A_{U}\right)$ the simplicial pair $\left(\mathcal{V}\left(X,X,U\right),\mathcal{V}\left(A,X,U\cap\left(A\times X\right)\right)\right)$.
If $U\subseteq V$, $X_{U}$ and $A_{U}$ are simplicial subsets of
$X_{V}$ and $A_{V}$, respectively. Let $i_{VU}\colon\left(X_{U},A_{U}\right)\hookrightarrow\left(X_{V},A_{V}\right)$
be the inclusion and let $p_{VU}:=H_{\bullet}\left(i_{VU};G\right)$.
Thus we have an inverse system $\left(\mathcal{E}_{X},H_{\bullet}\left(X_{U},A_{U};G\right),p_{VU}\right)$.
The uniform Vietoris homology of $\left(X,Y\right)$ can be redefined
as the inverse limit $\check{H}_{\bullet}^{u}\left(X,A;G\right):=\lim_{U\in\mathcal{E}_{X}}H_{\bullet}\left(X_{U},A_{U};G\right)$.
\end{rem}

\subsection{Microsimplicial homology theories}

We first recall the definition of McCord homology \cite{McC72}. McCord
homology is a homology theory on $\Top_{2}$ (or, more rigorously,
on its small full subcategory $\Top_{2}\cap\mathbb{U}$). Let $X$
be a standard topological space and $G$ an internal abelian group.
The \emph{monad} of $x\in X$ is the set
\[
\mu_{X}\left(x\right):=\bigcap\set{\ns{U}|x\in U\colon\text{open}}.
\]
We say that a member $\left(a_{0},\ldots,a_{p}\right)$ of $\ns{X^{p+1}}$
is a \emph{$p$\textendash microsimplex} if $\set{a_{0},\ldots,a_{p}}\subseteq\mu_{X}\left(x\right)$
holds for some $x\in X$. We denote by $C_{p}^{M}\left(X;G\right)$
the $G$-module that consists of all internal hyperfinite formal sums
of $p$\textendash microsimplices with coefficients in $G$. The boundary
map $\partial_{p}\colon C_{p}^{M}\left(X;G\right)\to C_{p-1}^{M}\left(X;G\right)$
is defined by
\[
\partial_{p}\left(a_{0},\ldots,a_{p}\right):=\sum_{i=0}^{p}\left(-1\right)^{i}\left(a_{0},\ldots,a_{i-1},a_{i+1},\ldots,a_{p}\right).
\]
It is obvious that $C_{\bullet}^{M}\left(X;G\right)$ is a chain complex,
called the \emph{McCord microchain complex}. Every standard continuous
map $f\colon X\to Y$ functorially induces a homomorphism $C_{\bullet}^{M}\left(f;G\right)\colon C_{\bullet}^{M}\left(X;G\right)\to C_{\bullet}^{M}\left(Y;G\right)$
defined by letting
\[
C_{p}^{M}\left(f;G\right)\left(a_{0},\ldots,a_{p}\right):=\left(\ns{f}\left(a_{0}\right),\ldots,\ns{f}\left(a_{p}\right)\right).
\]
Let $\left(X,A\right)$ be a standard topological pair. $C_{\bullet}^{M}\left(A;G\right)$
is a subchain complex of $C_{\bullet}^{M}\left(X;G\right)$. We define
\[
C_{\bullet}^{M}\left(X,A;G\right):=\frac{C_{\bullet}^{M}\left(X;G\right)}{C_{\bullet}^{M}\left(A;G\right)}.
\]
Let $f\colon\left(X,A\right)\to\left(Y,B\right)$ be a standard continuous
map. Since $C_{\bullet}^{M}\left(f;G\right)\colon C_{\bullet}^{M}\left(X;G\right)\to C_{\bullet}^{M}\left(Y;G\right)$
maps $C_{\bullet}^{M}\left(A;G\right)$ to $C_{\bullet}^{M}\left(B;G\right)$,
$f$ functorially induces a homomorphism $C_{\bullet}^{M}\left(f;G\right)\colon C_{\bullet}^{M}\left(X,A;G\right)\to C_{\bullet}^{M}\left(Y,B;G\right)$.
We define $H_{\bullet}^{M}\left(\cdot,\cdot;G\right):=H_{\bullet}C_{\bullet}^{M}\left(\cdot,\cdot;G\right)$,
where $H_{\bullet}$ is the ordinary homology functor of chain complexes.
The resulting functor is called \emph{McCord homology}.
\begin{rem}
Let $\left(X,A\right)$ be a standard topological pair. It may happen
that
\[
C_{p}^{M}\left(A;G\right)\neq\Set{\sum_{i}g_{i}\sigma_{i}\in C_{p}^{M}\left(X;G\right)|\sigma_{i}\subseteq\ns{A}\text{ for all }i}.
\]
For instance, consider the topological pair $\left(\mathbb{R},\mathbb{R}_{+}\right)$.
Given positive infinitesimals $a_{0},\ldots,a_{p}\in\ns{\mathbb{R}_{+}}$,
the $p$\textendash tuple $\left(a_{0},\ldots,a_{p}\right)$ is a
microsimplex of $\mathbb{R}$, although not a microsimplex of $\mathbb{R}_{+}$.
Because of this, the short exact sequence
\[
\xymatrix{0\ar[r] & C_{\bullet}^{M}\left(A;G\right)\ar[r] & C_{\bullet}^{M}\left(X;G\right)\ar[r] & C_{\bullet}^{M}\left(X,A;G\right)\ar[r] & 0}
\]
may not split. However, if $A$ is closed in $X$, the above two sets
are equal, and the short exact sequence splits.
\end{rem}

Next, recall the definition of Korppi homology \cite{Kor07,Kor12}.
Korppi homology is a homology theory on $\CR_{2}$ (or, more precisely,
on $\CR_{2}\cap\mathbb{U}$). Korppi's definition is similar to McCord's.
The only difference lies in the definition of the term `microsimplex'.
Let $X$ be a standard completely regular space and $G$ an internal
abelian group. Two points $x,y\in\ns{X}$ are said to be \emph{infinitely
close} (denoted by $x\sim_{X}y$) if for each normal cover $\lambda\in\mathcal{N}_{X}$,
$x$ and $y$ are $\ns{\lambda}$-close, i.e., $\set{x,y}\subseteq U$
for some $U\in\ns{\lambda}$. The \emph{normal monad} of $x\in\ns{X}$
is the set
\[
\mu_{X}^{n}\left(x\right):=\set{y\in\ns{X}|x\sim_{X}y}=\bigcap_{\lambda\in\mathcal{N}_{X}}\mathrm{St}\left(x,\ns{\lambda}\right),
\]
where $\mathrm{St}\left(-\right)$ denotes the star-neighbourhood.
For each $x\in X$, the normal monad of $x$ is equal to the monad
of $x$ \cite[Lemma 6-(2)]{Kor12}, while the normal monad makes sense
for $x\in\ns{X}\setminus X$. We say that a member $\left(a_{0},\ldots,a_{p}\right)$
of $\ns{X^{p+1}}$ is a \emph{$p$\textendash microsimplex} if $\set{a_{0},\ldots,a_{p}}\subseteq\mu_{X}^{n}\left(x\right)$
holds for some $x\in\ns{X}$, or equivalently, if $a_{i}\sim_{X}a_{j}$
for all $0\leq i,j\leq p$. The rest of the definition is the same
as McCord's one. We denote by $C_{\bullet}^{K}$ the Korppi microchain
complex functor and by $H_{\bullet}^{K}$ the Korppi homology functor.
\begin{rem}
\label{rem:splitting-of-K-hom}Let $\left(X,A\right)$ be a standard
completely regular pair. It is possible that $x\sim_{X}y$ but $x\nsim_{A}y$
for some $x,y\in\ns{A}$ (see \cite[Remark 8]{Kor12}). The short
exact sequence
\[
\xymatrix{0\ar[r] & C_{\bullet}^{K}\left(A;G\right)\ar[r] & C_{\bullet}^{K}\left(X;G\right)\ar[r] & C_{\bullet}^{K}\left(X,A;G\right)\ar[r] & 0}
\]
may not split. We say that $A$ is \emph{normally embedded} in $X$
if for every normal cover $\mu$ of $A$, there exists a normal cover
$\lambda$ of $X$ such that $\set{U\cap A|U\in\lambda}$ refines
$\mu$. If $A$ is normally embedded in $X$, then $\sim_{A}$ agrees
with $\sim_{X}$ on $\ns{A}$ \cite[Lemma 9]{Kor12}, and the above
sequence splits.
\end{rem}

Finally, we recall the definition of $\mu$-homology \cite{Ima16}.
$\mu$-homology is a homology theory on $\Unif_{2}$ (or, more precisely,
on $\Unif_{2}\cap\mathbb{U}$). The definition of $\mu$-homology
is similar to that of Korppi homology. Let $X$ be a standard uniform
space and $G$ an internal abelian group. Two points $x,y\in\ns{X}$
are said to be \emph{infinitely close} (denoted by $x\approx_{X}y$)
if for each uniform cover $\lambda\in\mathcal{U}_{X}$, $x$ and $y$
are $\ns{\lambda}$-close, or equivalently, if $x\mathrel{\ns{U}}y$
holds for any entourage $U\in\mathcal{E}_{X}$. The\emph{ uniform
monad} of $x\in\ns{X}$ is the set
\begin{align*}
\mu_{X}^{u}\left(x\right) & :=\set{y\in\ns{X}|x\approx_{X}y}\\
 & =\bigcap_{\lambda\in\mathcal{U}_{X}}\mathrm{St}\left(x,\ns{\lambda}\right)\\
 & =\bigcap_{U\in\mathcal{E}_{X}}\ns{U}\left[x\right],
\end{align*}
where $U\left[x\right]:=\set{y\in X|\left(x,y\right)\in U}$. (We
will also use the notation $U\left[A\right]:=\set{y\in X|\exists x\in A.\left(x,y\right)\in U}$
throughout.) The uniform monad of $x$ coincides with the monad of
$x$ for $x\in X$. We say that a member $\left(a_{0},\ldots,a_{p}\right)$
of $\ns{X^{p+1}}$ is a \emph{$p$\textendash microsimplex} if $\set{a_{0},\ldots,a_{p}}\subseteq\mu_{X}^{u}\left(x\right)$
holds for some $x\in\ns{X}$, or equivalently, if $a_{i}\approx_{X}a_{j}$
for all $0\leq i,j\leq p$. The rest of the definition is the same
as McCord's and Korppi's. We denote by $C_{\bullet}^{\mu}$ the $\mu$-microchain
complex functor and by $H_{\bullet}^{\mu}$ the $\mu$-homology functor.
\begin{rem}
In contrast to the previous case (\prettyref{rem:splitting-of-K-hom}),
$\approx_{A}$ agrees with $\approx_{X}$ on $\ns{A}$ for each standard
uniform pair $\left(X,A\right)$. Hence the short exact sequence
\[
\xymatrix{0\ar[r] & C_{\bullet}^{\mu}\left(A;G\right)\ar[r] & C_{\bullet}^{\mu}\left(X;G\right)\ar[r] & C_{\bullet}^{\mu}\left(X,A;G\right)\ar[r] & 0}
\]
always splits (see \cite[Proposition 2]{Ima16}).
\end{rem}

The excision property of $\mu$-homology is not proved in the preceding
paper \cite{Ima16}, whilst a weak form of the excision is (see \cite[Proposition 3]{Ima16}).
In the rest of this section, we prove the \emph{full} excision property
of $\mu$-homology.
\begin{defn}
Let $X$ be a uniform space. Let $A$ and $B$ be subsets of $X$.
We say that $A$ is \emph{strongly contained in $B$} (and we write
$A\Subset B$) if $\mathrm{St}\left(A,\lambda\right)\subseteq B$
for some $\lambda\in\mathcal{U}_{X}$.
\end{defn}

\begin{lem}
\label{lem:n-c-of-to-be-strong-inclusion}Let $X$ be a standard uniform
space. Let $A$ and $B$ be subsets of $X$. Then $A\Subset B$ if
and only if $\mu_{X}^{u}\left(\ns{A}\right)\subseteq\ns{B}$, where
$\mu_{X}^{u}\left(\ns{A}\right):=\bigcup_{a\in\ns{A}}\mu_{X}^{u}\left(a\right)$.
\end{lem}

\begin{proof}
Suppose that $\mu_{X}^{u}\left(\ns{A}\right)\subseteq\ns{B}$. By
weak saturation, we can find an $\lambda\in\ns{\mathcal{U}_{X}}$
such that $\lambda$ refines $\ns{\nu}$ for all $\nu\in\mathcal{U}_{X}$.
Then $\mathrm{St}\left(\ns{A},\lambda\right)\subseteq\mu_{X}^{u}\left(\ns{A}\right)\subseteq\ns{B}$.
By transfer, we see that $A\Subset B$. Conversely, suppose $A\Subset B$.
Let $\lambda\in\mathcal{U}_{X}$ be with $\mathrm{St}\left(A,\lambda\right)\subseteq B$.
Then, by transfer, we have that $\mu_{X}^{u}\left(\ns{A}\right)\subseteq\ns{\left(\mathrm{St}\left(A,\lambda\right)\right)}\subseteq\ns{B}$.
\end{proof}
\begin{prop}
\label{prop:Excision-of-mu-hom}Let $X$ be a standard uniform space.
Let $A$ and $B$ be subsets of $X$. If $X\setminus A\Subset B$
(or $X\setminus B\Subset A$), then the inclusion map $i\colon\left(A,A\cap B\right)\hookrightarrow\left(X,B\right)$
induces an isomorphism $H_{\bullet}^{\mu}\left(i;G\right)\colon H_{\bullet}^{\mu}\left(A,A\cap B;G\right)\cong H_{\bullet}^{\mu}\left(X,B;G\right)$.
\end{prop}

\begin{proof}
The proof is a slight modification of the proof of \cite[Proposition 3]{Ima16}.
It suffices to show that the following two inclusions hold:
\begin{enumerate}
\item $C_{p}^{\mu}\left(A;G\right)\cap C_{p}^{\mu}\left(B;G\right)\subseteq C_{p}^{\mu}\left(A\cap B;G\right)$,
\item $C_{p}^{\mu}\left(X;G\right)\subseteq C_{p}^{\mu}\left(A;G\right)+C_{p}^{\mu}\left(B;G\right)$.
\end{enumerate}
The first inclusion is trivial. We only need to prove that each microsimplex
$\sigma$ of $X$ is contained in either $\ns{A}$ or $\ns{B}$. Suppose
that $\sigma$ is not contained in $\ns{A}$. Therefore, $\sigma$
intersects $\ns{\left(X\setminus A\right)}$. All the vertices of
$\sigma$ are in $\mu_{X}^{u}\left(\ns{\left(X\setminus A\right)}\right)$.
By \prettyref{lem:n-c-of-to-be-strong-inclusion}, $\sigma$ is contained
in $\ns{B}$.

\end{proof}
\begin{cor}
Let $X$ be a standard uniform space. Let $A$ and $U$ be subsets
of $X$. If $U\Subset A$, then the inclusion map $i\colon\left(X\setminus U,A\setminus U\right)\hookrightarrow\left(X,A\right)$
induces an isomorphism $H_{\bullet}^{\mu}\left(i;G\right)\colon H_{\bullet}^{\mu}\left(X\setminus U,A\setminus U;G\right)\cong H_{\bullet}^{\mu}\left(X,A;G\right)$.
\end{cor}

\begin{proof}
By applying \prettyref{prop:Excision-of-mu-hom} to the triad $\left(X,X\setminus U,A\right)$,
we obtain the desired result.
\end{proof}

\section{\label{sec:Rel-to-Korppi-homology}Relationship among the microsimplicial
homology theories}

In the first half of this section, we discuss the relationship among
the microsimplicial homology theories. In \prettyref{subsec:McCord-homology-and-mu-homology},
we prove the equality of McCord homology and $\mu$-homology on the
category of compact topological spaces. In \prettyref{subsec:Korppi-homology-and-mu-homology},
we prove the equality of Korppi homology and $\mu$-homology on the
category of fine uniform spaces. $\mu$-homology inherits many properties
from Korppi homology. In the rest of this section, we present two
examples of such properties, the embeddability of the Vietoris-type
homology (\prettyref{subsec:embedding-of-V-hom}) and the continuity
with respect to resolutions (\prettyref{subsec:Continuity-of-mu-homology}).

\subsection{\label{subsec:McCord-homology-and-mu-homology}McCord homology and
$\mu$-homology}

In \cite{Ima16}, the author stated (without detailed proof) that
$\mu$-homology and McCord homology are completely the same for all
compact uniform spaces. Here we give a detailed proof.

Let $U\colon\Unif\to\CR$ be the forgetful (topologisation) functor,
where $\Unif$ is the category of uniform spaces, and $\CR$ is the
category of completely regular spaces.
\begin{lem}
Let $G$ be an internal abelian group. Let $X$ be a standard uniform
space. Then $C_{\bullet}^{M}\left(UX;G\right)\subseteq C_{\bullet}^{\mu}\left(X;G\right)$.
If $X$ is compact, $C_{\bullet}^{\mu}\left(X;G\right)\subseteq C_{\bullet}^{M}\left(UX;G\right)$.
\end{lem}

\begin{proof}
Let $\left(a_{0},\ldots,a_{p}\right)\in\ns{X^{p+1}}$ be a McCord
microsimplex. Choose an $x\in X$ such that $\set{a_{0},\ldots,a_{p}}\subseteq\mu_{UX}\left(x\right)$.
Since $\mu_{UX}\left(x\right)=\mu_{X}^{u}\left(x\right)$, $\left(a_{0},\ldots,a_{p}\right)$
is a $\mu$-microsimplex. It follows that $C_{\bullet}^{M}\left(UX;G\right)\subseteq C_{\bullet}^{\mu}\left(X;G\right)$.

Assume that $X$ is compact. Let $\left(a_{0},\ldots,a_{p}\right)\in\ns{X^{p+1}}$
be a $\mu$-microsimplex. By the nonstandard characterisation of compactness
\cite[Theorem 4.1.13]{Rob66}, there exists an $x\in X$ such that
$a_{0}\in\mu_{UX}\left(x\right)=\mu_{X}^{u}\left(x\right)$. Since
$a_{i}\approx_{X}a_{0}\approx_{X}x$, it follows that $a_{i}\in\mu_{UX}\left(x\right)$
for all $0\leq i\leq p$. Therefore, $\left(a_{0},\ldots,a_{p}\right)$
is a McCord microsimplex. Hence $C_{\bullet}^{\mu}\left(X;G\right)\subseteq C_{\bullet}^{M}\left(UX;G\right)$.
\end{proof}
\begin{prop}
\label{prop:MF=00003DMu}Let $G$ be an internal abelian group. Let
$\left(X,A\right)$ be a standard compact uniform pair. Then, we have
that $C_{\bullet}^{M}\left(UX,A;G\right)=C_{\bullet}^{\mu}\left(X,A;G\right)$
and $H_{\bullet}^{M}\left(UX,A;G\right)=H_{\bullet}^{\mu}\left(X,A;G\right)$.
\end{prop}

This equality does not hold for some non-compact uniform spaces (see
\cite[Example 18]{Ima16}).

\subsection{\label{subsec:Korppi-homology-and-mu-homology}Korppi homology and
$\mu$-homology}

Let $F\colon\CR\to\Unif$ be the left adjoint functor of $U\colon\Unif\to\CR$.
More specifically, given a completely regular space $X$, $FX$ is
the uniform space whose underlying set is the same as $X$ and whose
uniformity is the finest uniformity compatible with the topology of
$X$.
\begin{lem}
Let $X$ be a standard completely regular space. For any $x,y\in\ns{X}$,
$x\sim_{X}y$ if and only if $x\approx_{FX}y$. Hence $\mu_{X}^{n}=\mu_{FX}^{u}$.
\end{lem}

\begin{proof}
Immediately from the fact that every uniform (open) cover of $FX$
is a normal cover of $X$, and vice versa (see \cite[Theorem 20]{Isb64}).
\end{proof}
The definition of $\mu$-homology is very similar to that of Korppi
homology. Recall that the only difference lies in the definition of
the term `infinitely close to'. The following equivalence is now obvious.
\begin{prop}
\label{prop:K=00003DMuF}Let $G$ be an internal abelian group. Let
$\left(X,A\right)$ be a standard completely regular pair such that
$A$ is normally embedded in $X$. Then, we have that $C_{\bullet}^{K}\left(X,A;G\right)=C_{\bullet}^{\mu}\left(FX,A;G\right)$
and $H_{\bullet}^{K}\left(X,A;G\right)=H_{\bullet}^{\mu}\left(FX,A;G\right)$.
\end{prop}

\begin{cor}[{\cite[Remark 46]{Kor12}}]
Let $G$ be an internal abelian group. Let $\left(X,A\right)$ be
a standard compact completely regular pair. Then, we have that $C_{\bullet}^{K}\left(X,A;G\right)=C_{\bullet}^{M}\left(X,A;G\right)$
and $H_{\bullet}^{K}\left(X,A;G\right)=H_{\bullet}^{M}\left(X,A;G\right)$.
\end{cor}

Because of \prettyref{prop:K=00003DMuF}, $\mu$-homology is a generalisation
of Korppi homology from fine uniform spaces to general uniform spaces.
$\mu$-homology inherits many properties from Korppi homology (see
the next two subsections).
\begin{rem}
The equality between Korppi and $\mu$-homology does not hold for
some non-fine uniform spaces (e.g. the real line without the origin
equipped with the usual uniformity). In addition, the equality between
McCord and Korppi homology does not hold for some non-compact completely
regular spaces. The topologist's sine curve and the deleted comb space
(\prettyref{fig:ex-of-locally-complicated-spaces}) give counterexamples.
\begin{figure}
\subfloat[The topologist's sine curve]{\begin{tikzpicture}
\draw (5/1, 0) -- (5/2, 1);
\draw (5/2, 1) -- (5/4, -1);
\foreach \n in {2, ..., 25}
	\draw ({5 / (2 * \n)}, -1) -- ({5 / (2 * \n + 1)}, 1) -- ({5 / (2 * \n + 2)}, -1);
\fill (0, 0) circle[radius=1pt] node [left=2pt] {$\left(0,0\right)$};
\end{tikzpicture}

}

\subfloat[The deleted comb space]{\begin{tikzpicture}
\foreach \n in {1, ..., 25}
	\draw ({5 / \n}, -1) -- ({5 / \n}, 1);
\draw (0, -1) -- (5, -1);
\fill (0, -1) circle[radius=1pt] node [left=2pt] {$\left(0,0\right)$};
\fill (0, 1) circle[radius=1pt] node [left=2pt] {$\left(0,1\right)$};
\end{tikzpicture}}\caption{\label{fig:ex-of-locally-complicated-spaces}}
\end{figure}
\end{rem}

\subsection{\label{subsec:embedding-of-V-hom}Natural embeddings of Vietoris-type
into microsimplicial homology}

Korppi proved that Vietoris homology with standard coefficients can
be embedded into Korppi homology with nonstandard coefficients for
all paracompact spaces \cite[Theorem 76]{Kor12}. Actually, Korppi
homology is better related with \emph{normal} Vietoris homology than
with Vietoris homology. For example, normal Vietoris homology can
be embedded into Korppi homology without any extra condition.
\begin{thm}
Let $G$ be a standard abelian group. Let $\left(X,A\right)$ be a
standard completely regular pair such that $A$ is normally embedded
in $X$. Then there exists a monomorphism $\check{H}_{\bullet}^{n}\left(X,A;G\right)\to H_{\bullet}^{K}\left(X,A;\ns{G}\right)$
natural in $\left(X,A\right)$.
\end{thm}

\begin{proof}
Let $C_{\bullet}$ denote the ordinary chain complex functor of simplicial
pairs. For each $\lambda\in\mathcal{N}_{X}$, let $i_{\lambda}\colon H_{\bullet}\left(X_{\lambda},A_{\lambda};G\right)\to\ns{\left(H_{\bullet}\left(X_{\lambda},A_{\lambda};G\right)\right)}$
be the elementary embedding. These maps induce a limiting map $i_{X,A}:=\lim_{\lambda\in\mathcal{N}_{X}}i_{\lambda}\colon\lim_{\lambda\in\mathcal{N}_{X}}H_{\bullet}\left(X_{\lambda},A_{\lambda};G\right)\to\lim_{\lambda\in\mathcal{N}_{X}}\ns{\left(H_{\bullet}\left(X_{\lambda},A_{\lambda};G\right)\right)}$.
It is easy to see that $i_{X,A}$ is a monomorphism natural in $\left(X,A\right)$.
The domain $\lim_{\lambda\in\mathcal{N}_{X}}H_{\bullet}\left(X_{\lambda},A_{\lambda};G\right)$
is precisely the same as $\check{H}_{\bullet}^{n}\left(X,A;G\right)$.
By \cite[Theorem 48]{Kor12}, the codomain is
\[
\begin{aligned}\lim_{\lambda\in\mathcal{N}_{X}}\ns{\left(H_{\bullet}\left(X_{\lambda},A_{\lambda};G\right)\right)} & =\lim_{\lambda\in\mathcal{N}_{X}}\ns{\left(H_{\bullet}\left(C_{\bullet}\left(X_{\lambda},A_{\lambda};G\right)\right)\right)}\\
 & =H_{\bullet}\lim_{\lambda\in\mathcal{N}_{X}}\ns{\left(C_{\bullet}\left(X_{\lambda},A_{\lambda};G\right)\right)}\\
 & =H_{\bullet}\bigcap_{\lambda\in\mathcal{N}_{X}}\ns{\left(C_{\bullet}\left(X_{\lambda},A_{\lambda};G\right)\right)}\\
 & =H_{\bullet}C_{\bullet}^{K}\left(X,A;\ns{G}\right)\\
 & =H_{\bullet}^{K}\left(X,A;\ns{G}\right).
\end{aligned}
\]
Hence $i_{X,A}$ is a desired natural monomorphism $\check{H}_{\bullet}^{n}\left(X,A;G\right)\to H_{\bullet}^{K}\left(X,A;\ns{G}\right)$.
\end{proof}
One can prove the following uniform analogue in the same way.
\begin{thm}
Let $G$ be a standard abelian group. Let $\left(X,A\right)$ be a
standard uniform pair. Then there exists a monomorphism $\check{H}_{\bullet}^{u}\left(X,A;G\right)\to H_{\bullet}^{\mu}\left(X,A;\ns{G}\right)$
natural in $\left(X,A\right)$.
\end{thm}

\subsection{\label{subsec:Continuity-of-mu-homology}Continuity of $\mu$-homology
with respect to uniform resolutions}

Korppi homology is continuous with respect to resolutions (\cite[Theorem 71]{Kor12}).
Analogously, $\mu$-homology is continuous with respect to uniform
resolutions.
\begin{defn}
Let $\bm{X}:=\left(I,X_{i},\pi_{ij}\right)$ be an inverse system
of uniform spaces and let $\pi\colon X\to\bm{X}$ be a cone over $\bm{X}$.
Then $\bm{X}$ (together with $\pi$) is called a \emph{uniform resolution}
of $X$ if the following conditions hold:
\begin{description}
\item [{(UR1)}] for each $\lambda\in\mathcal{U}_{X}$ there exist an $i\in I$
and a $\mu\in\mathcal{U}_{X_{i}}$ such that $\pi_{i}^{-1}\mu$ refines
$\lambda$;
\item [{(UR2)}] for each $i\in I$ and each $\lambda\in\mathcal{U}_{X_{i}}$
there exists an $j\in I$ such that $\pi_{ij}\left(X_{j}\right)\subseteq\mathrm{St}\left(\pi_{i}\left(X\right),\lambda\right)$.
\end{description}
Let $\left(\bm{X},\bm{A}\right):=\left(I,X_{i},A_{i},\pi_{ij}\right)$
be an inverse system of uniform pairs and let $\pi\colon\left(X,A\right)\to\left(\bm{X},\bm{A}\right)$
be a cone over $\left(\bm{X},\bm{A}\right)$. Then $\left(\bm{X},\bm{A}\right)$
is called a uniform resolution of $\left(X,A\right)$ if $\left(I,X_{i},\pi_{ij}\right)$
and $\left(I,A_{i},\pi_{ij}\right)$ are uniform resolutions of $X$
and $A$, respectively.

\end{defn}

\begin{thm}
\label{thm:continuity-of-mu-homology}Let $\left(\bm{X},\bm{A}\right):=\left(I,X_{i},A_{i},\pi_{ij}\right)$
be a standard inverse system of uniform pairs. Let $\pi\colon\left(X,A\right)\to\left(\bm{X},\bm{A}\right)$
be a standard cone over $\left(\bm{X},\bm{A}\right)$. If $\left(\bm{X},\bm{A}\right)$
is a uniform resolution of $\left(X,A\right)$, then $\pi$ induces
an isomorphism $H_{\bullet}^{\mu}\left(X,A;G\right)\cong\lim_{i\in I}H_{\bullet}^{\mu}\left(X_{i},A_{i};G\right)$.
\end{thm}

The proof is completely analogous to Korppi's. For example, we should
replace \cite[Lemma 60]{Kor12} and \cite[Lemma 61]{Kor12} with the
following two lemmas.
\begin{defn}
Let $\bm{X}:=\left(I,X_{i},\pi_{ij}\right)$ be a standard inverse
system of uniform spaces. We denote $J:=\set{j\in\ns{I}|i\leq j\text{ for all }i\in I}$.
Note that $J$ is nonempty by weak saturation. For $x,y\in\ns{X_{j}}\ \left(j\in J\right)$,
we write $x\approx_{\bm{X}}y$ if $\ns{\pi_{ij}}\left(x\right)\approx_{X_{i}}\ns{\pi_{ij}}\left(y\right)$
holds for any $i\in I$.
\end{defn}

\begin{lem}
Let $\bm{X}:=\left(I,X_{i},\pi_{ij}\right)$ be a standard inverse
system of uniform spaces and let $\pi\colon X\to\bm{X}$ be a standard
cone over $\bm{X}$. The following are equivalent:
\begin{enumerate}
\item (UR1);
\item for any $x,y\in\ns{X}$, $x\approx_{X}y$ if and only if $\ns{\pi_{i}}\left(x\right)\approx_{X_{i}}\ns{\pi_{i}}\left(y\right)$
for all $i\in I$;
\item for any $x,y\in\ns{X}$, $x\approx_{X}y$ if and only if $\ns{\pi_{j}}\left(x\right)\approx_{\bm{X}}\ns{\pi_{j}}\left(y\right)$
for all $j\in J$;
\item for any $x,y\in\ns{X}$, $x\approx_{X}y$ if and only if $\ns{\pi_{j}}\left(x\right)\approx_{\bm{X}}\ns{\pi_{j}}\left(y\right)$
for some $j\in J$.
\end{enumerate}
\end{lem}

\begin{proof}
\begin{description}
\item [{$\left(1\right)\Rightarrow\left(2\right)$}] Suppose that $x\not\approx_{X}y$.
We can find a $\lambda\in\mathcal{U}_{X}$ such that $x$ and $y$
are not $\ns{\lambda}$-near. There exist an $i\in I$ and a $\mu\in\mathcal{U}_{X_{i}}$
such that $\pi_{i}^{-1}\mu$ refines $\lambda$. $x$ and $y$ are
not $\ns{\left(\pi_{i}^{-1}\mu\right)}$-near. Therefore $\ns{\pi_{i}}\left(x\right)$
and $\ns{\pi_{i}}\left(y\right)$ are not $\ns{\mu}$-near. Hence
$\ns{\pi_{i}}\left(x\right)\not\approx_{X_{i}}\ns{\pi_{i}}\left(y\right)$.
The reverse direction immediately follows from the uniform continuity
of $\pi_{i}$.
\item [{$\left(2\right)\Rightarrow\left(1\right)$}] Suppose, on the contrary,
that there exists a $\lambda\in\mathcal{U}_{X}$ such that for any
$i\in I$ and any uniform cover $\mu\in\mathcal{U}_{X_{i}}$, $\pi_{i}^{-1}\mu$
does not refine $\lambda$. Let $\lambda'$ be a uniform star-refinement
of $\lambda$.

Let $i\in I$ and $\mu\in\mathcal{U}_{X_{i}}$. Since $\pi_{i}^{-1}\mu$
does not refine $\lambda$, there exists a $V\in\pi_{i}^{-1}\mu$
such that $V$ is not contained in any member of $\lambda$. Clearly
$V\neq\varnothing$. Choose an $x_{i,\mu}\in V$. We have that $V\nsubseteq\mathrm{St}\left(x_{i,\mu},\lambda'\right)$.
Choose a $y_{i,\mu}\in V\setminus\mathrm{St}\left(x_{i,\mu},\lambda'\right)$.
Then, $x_{i,\mu}$ and $y_{i,\mu}$ are $\pi_{i}^{-1}\mu$-near but
not $\lambda'$-near.

Now, let $P$ be a finite set of all pairs $\left(i,\mu\right)$ such
that $i\in I$ and $\mu\in\mathcal{U}_{X_{i}}$. Fix an upper bound
$i'\in I$ of all $i$ with $\left(i,\mu\right)\in P$, and a common
refinement $\mu'$ of all $\pi_{ii'}^{-1}\mu$ with $\left(i,\mu\right)\in P$.
We can find $x_{P},y_{P}\in X$ such that $x_{P}$ and $y_{P}$ are
$\pi_{i'}^{-1}\mu'$-near but not $\lambda'$-near. $x_{P}$ and $y_{P}$
are $\pi_{i}^{-1}\mu$-near for all $\left(i,\mu\right)\in P$. By
weak saturation, there exist $x,y\in\ns{X}$ such that $x$ and $y$
are $\ns{\left(\pi_{i}^{-1}\mu\right)}$-near for all $i\in I$ and
$\mu\in\mathcal{U}_{X_{i}}$ but not $\ns{\lambda'}$-near. Hence
$\ns{\pi_{i}}\left(x\right)\approx_{X_{i}}\ns{\pi_{i}}\left(y\right)$
but $x\not\approx_{X}y$.
\item [{$\left(2\right)\Rightarrow\left(3\right)$}] Suppose $x\approx_{X}y$.
Let $j\in J$. Since $\pi_{i}$ is uniformly continuous, we have that
$\ns{\pi_{ij}}\left(\ns{\pi_{j}}\left(x\right)\right)=\ns{\pi_{i}}\left(x\right)\approx_{X_{i}}\ns{\pi_{i}}\left(y\right)=\ns{\pi_{ij}}\left(\ns{\pi_{j}}\left(y\right)\right)$
for all $i\in I$. Hence $\ns{\pi_{j}}\left(x\right)\approx_{\bm{X}}\ns{\pi_{j}}\left(y\right)$.
Next, suppose that $\ns{\pi_{j}}\left(x\right)\approx_{\bm{X}}\ns{\pi_{j}}\left(y\right)$
for all $j\in J$. Fix a $j_{0}\in J$. By the definition of $\approx_{\bm{X}}$,
for any $i\in I$ we have that
\begin{align*}
\ns{\pi_{i}}\left(x\right) & =\ns{\pi_{ij_{0}}}\left(\ns{\pi_{j_{0}}}\left(x\right)\right)\\
 & \approx_{X_{i}}\ns{\pi_{ij_{0}}}\left(\ns{\pi_{j_{0}}}\left(y\right)\right)\\
 & =\ns{\pi_{i}}\left(y\right).
\end{align*}
By (2) we have $x\approx_{X}y$.
\item [{$\left(3\right)\Rightarrow\left(4\right)$}] Suppose that $\ns{\pi_{j_{0}}}\left(x\right)\approx_{\bm{X}}\ns{\pi_{j_{0}}}\left(y\right)$
for some $j_{0}\in J$. Let $j\in J$. For any $i\in I$ we have that
\begin{align*}
\ns{\pi_{ij}}\left(\ns{\pi_{j}}\left(x\right)\right) & =\ns{\pi_{i}}\left(x\right)\\
 & =\ns{\pi_{ij_{0}}}\left(\ns{\pi_{j_{0}}}\left(x\right)\right)\\
 & \approx_{X_{i}}\ns{\pi_{ij_{0}}}\left(\ns{\pi_{j_{0}}}\left(y\right)\right)\\
 & =\ns{\pi_{i}}\left(y\right)\\
 & =\ns{\pi_{ij}}\left(\ns{\pi_{j}}\left(y\right)\right).
\end{align*}
Hence $\ns{\pi_{j}}\left(x\right)\approx_{\bm{X}}\ns{\pi_{j}}\left(y\right)$.
By (3) we have $x\approx_{X}y$. The reverse direction is trivial.
\item [{$\left(4\right)\Rightarrow\left(2\right)$}] Suppose that $\ns{\pi_{i}}\left(x\right)\approx_{X_{i}}\ns{\pi_{i}}\left(y\right)$
for all $i\in I$. Fix a $j_{0}\in J$. We have that $\ns{\pi_{j_{0}i}}\left(\ns{\pi_{j_{0}}}\left(x\right)\right)=\ns{\pi_{i}}\left(x\right)\approx_{X_{i}}\ns{\pi_{i}}\left(y\right)=\ns{\pi_{j_{0}i}}\left(\ns{\pi_{j_{0}}}\left(y\right)\right)$
for all $i\in I$. Hence $\ns{\pi_{j_{0}}}\left(x\right)\approx_{\bm{X}}\ns{\pi_{j_{0}}}\left(y\right)$.
By (4) we have $x\approx_{X}y$. The reverse direction immediately
follows from the uniform continuity of $\pi_{i}$.\qedhere
\end{description}
\end{proof}
\begin{lem}
Let $\bm{X}:=\left(I,X_{i},\pi_{ij}\right)$ be a standard inverse
system of uniform spaces and let $\pi\colon X\to\bm{X}$ be a standard
cone over $\bm{X}$. The following are equivalent:
\begin{enumerate}
\item (UR2);
\item for any $j\in J$ and any $i\in I$, every $x\in\ns{\pi_{ij}}\left(\ns{X_{j}}\right)$
is $\approx_{X_{i}}$-near to some $x'\in\ns{\pi_{i}}\left(\ns{X}\right)$;
\item for some $j\in J$ and any $i\in I$, every $x\in\ns{\pi_{ij}}\left(\ns{X_{j}}\right)$
is $\approx_{X_{i}}$-near to some $x'\in\ns{\pi_{i}}\left(\ns{X}\right)$;
\item for any $j\in J$, every $x\in\ns{X_{j}}$ is $\approx_{\bm{X}}$-near
to some $x'\in\ns{\pi_{j}}\left(\ns{X}\right)$;
\item for some $j\in J$, every $x\in\ns{X_{j}}$ is $\approx_{\bm{X}}$-near
to some $x'\in\ns{\pi_{j}}\left(\ns{X}\right)$.
\end{enumerate}
\end{lem}

\begin{proof}
$\left(2\right)\Rightarrow\left(3\right)$ and $\left(4\right)\Rightarrow\left(5\right)$
are trivial.
\begin{description}
\item [{$\left(1\right)\Rightarrow\left(2\right)$}] Let $j\in J$, $i\in I$
and $x\in\ns{\pi_{ji}}\left(\ns{X_{j}}\right)$. Let $P$ be a finite
subset of $\mathcal{U}_{X_{i}}$. Let $\lambda'$ be a common refinement
of the members of $P$. Choose an $i'\geq i$ such that $\pi_{ii'}\left(X_{i'}\right)\subseteq\mathrm{St}\left(\pi_{i}\left(X\right),\lambda'\right)$.
Since $\ns{\pi_{ij}}\left(\ns{X_{j}}\right)\subseteq\ns{\left(\pi_{ii'}\left(X_{i'}\right)\right)}\subseteq\ns{\left(\mathrm{St}\left(\pi_{i}\left(X\right),\lambda'\right)\right)}$,
$x$ is $\ns{\lambda'}$-near to some $x_{P}\in\ns{\pi_{i}}\left(\ns{X}\right)$.
Such an $x_{P}$ is $\ns{\lambda}$-near to $x$ for all $\lambda\in P$.
By saturation, there exists an $x'\in\ns{\pi_{i}}\left(\ns{X}\right)$
such that $x$ and $x'$ are $\ns{\lambda}$-near for all $\lambda\in\mathcal{U}_{X_{i}}$,
i.e., $x\approx_{X_{i}}x'$.
\item [{$\left(2\right)\Rightarrow\left(4\right)$ and $\left(3\right)\Rightarrow\left(5\right)$}] Let
$j\in J$. Suppose that for any $i\in I$ every $x\in\ns{\pi_{ij}}\left(\ns{X_{j}}\right)$
is $\approx_{X_{i}}$-near to some $x'\in\ns{\pi_{i}}\left(\ns{X}\right)$.
Let $x\in\ns{X_{j}}$. Given a finite set $P$ of all pairs $\left(i,\lambda\right)$
such that $i\in I$ and $\lambda\in\mathcal{U}_{X_{i}}$, choose an
upper bound $i'\in I$ of all $i$ with $\left(i,\lambda\right)\in P$
and a common refinement $\lambda'$ of all $\pi_{ii'}^{-1}\lambda$
with $\left(i,\lambda\right)\in P$. Since $\ns{\pi_{i'j}}\left(x\right)$
is $\ns{\lambda'}$-near to some points in $\ns{\pi_{i'}}\left(\ns{X}\right)$,
$x$ is $\ns{\pi_{i'j}^{-1}}\left(\ns{\lambda'}\right)$-near to some
$x_{P}\in\prescript{\ast}{}{\pi_{j}}\left(\prescript{\ast}{}{X}\right)$.
Such an $x_{P}$ is $\ns{\pi_{ij}^{-1}}\left(\ns{\lambda}\right)$-near
to $x$ for all $\left(i,\lambda\right)\in P$. By saturation, there
exists an $x'\in\ns{\pi_{j}}\left(\ns{X}\right)$ such that $x$ and
$x'$ are $\ns{\pi_{ij}}^{-1}\left(\ns{\lambda}\right)$-near for
all $\lambda\in\mathcal{U}_{X_{i}}$ and $i\in I$, i.e., $x\approx_{\bm{X}}x'$.
\item [{$\left(5\right)\Rightarrow\left(1\right)$}] Suppose that (1) does
not hold, i.e., there exist an $i\in I$ and a $\lambda\in\mathcal{U}_{X_{i}}$
such that $\pi_{ii'}\left(X_{i'}\right)\nsubseteq\mathrm{St}\left(\pi_{i}\left(X\right),\lambda\right)$
for all $i'\in I$. For each $i'\geq i$ choose an $x_{i'}\in X_{i'}$
such that $\pi_{ii'}\left(x_{i'}\right)\notin\mathrm{St}\left(\pi_{i}\left(X\right),\lambda\right)$.
Let $j\in J$. By transfer, we have that $\ns{x_{j}}\in\ns{X_{j}}$
and $\ns{\pi_{ij}}\left(\ns{x_{j}}\right)\notin\ns{\left(\mathrm{St}\left(\pi_{i}\left(X\right),\lambda\right)\right)}$.
It follows that $\ns{x_{j}}$ is not $\approx_{\bm{X}}$-near to any
$x'\in\ns{\pi_{j}}\left(\ns{X}\right)$.\qedhere
\end{description}
\end{proof}
Observe that the modified lemmas and their proofs are completely similar
to the original ones. The remaining part of \cite[Chapter 10, 11, 13]{Kor12}
can be transferred as well. Thus we obtain the proof of \prettyref{thm:continuity-of-mu-homology}.

\section{\label{sec:Rel-between-uniform-Vietoris-and-mu-homology}Relationship
between the Vietoris-type and microsimplicial homology theories}

In this section, we show that the microsimplicial homology theories
are isomorphic to the Vietoris-type homology theories under certain
compactness conditions.

\subsection{The compact case}

Garavaglia \cite{Gar78} showed that McCord homology is isomorphic
to Vietoris homology for all standard compact spaces. Using this result,
we can prove the following.
\begin{prop}
\label{prop:The-compact-case}Let $G$ be a standard abelian group.
Let $\left(X,A\right)$ be a standard compact uniform pair. Then,
$H_{\bullet}^{\mu}\left(X,A;\ns{G}\right)\cong\check{H}_{\bullet}^{u}\left(X,A;\ns{G}\right)$.
\end{prop}

\begin{proof}
We know that $H_{\bullet}^{\mu}\left(X,A;\ns{G}\right)=H_{\bullet}^{M}\left(UX,A;\ns{G}\right)$.
By \cite[Theorem 9]{Kor10}, $H_{\bullet}^{M}\left(UX,A;\ns{G}\right)\cong\check{H}_{\bullet}\left(UX,A;\ns{G}\right)$.
By Lebesgue's number lemma, $\check{H}_{\bullet}\left(UX,A;\ns{G}\right)\cong\check{H}_{\bullet}^{u}\left(X,A;\ns{G}\right)$.
Combining these isomorphisms gives $H_{\bullet}^{\mu}\left(X,A;\ns{G}\right)\cong\check{H}_{\bullet}^{u}\left(X,A;\ns{G}\right)$.
\end{proof}
\begin{cor}
Let $G$ be a standard abelian group. Let $\left(X,A\right)$ be a
standard compact completely regular pair. Then, $H_{\bullet}^{K}\left(X,A;\ns{G}\right)\cong\check{H}_{\bullet}^{n}\left(X,A;\ns{G}\right)$.
\end{cor}

This proof depends on the fact that uniform Vietoris homology is isomorphic
to Vietoris homology for all compact uniform spaces. In the non-compact
case, uniform Vietoris homology may not be isomorphic to Vietoris
homology. For instance, consider the quotient space $\mathbb{Q}/\mathbb{Z}$.
Since $\mathbb{Q}/\mathbb{Z}$ has disjoint open covers as fine as
one likes, the $1$-st Vietoris homology of $\mathbb{Q}/\mathbb{Z}$
vanishes. On the other hand, $\mathbb{Q}/\mathbb{Z}$ and $\mathbb{R}/\mathbb{Z}$
have the same uniform Vietoris homology. In particular, the $1$-st
uniform Vietoris homology of $\mathbb{Q}/\mathbb{Z}$ is isomorphic
to $G$, and hence does not vanish (except for the trivial case $G=0$).
We note that $\mathbb{Q}/\mathbb{Z}$ is a dense subspace of $\mathbb{R}/\mathbb{Z}$.

\subsection{Nonstandard continuity and homotopy}
\begin{defn}
Let $X$ and $Y$ be standard uniform spaces. A map $f\colon\ns{X}\to\ns{Y}$
is said to be \emph{S-continuous} at $x\in\ns{X}$ if $f\left(\mu_{X}^{u}\left(x\right)\right)\subseteq\mu_{Y}^{u}\left(f\left(x\right)\right)$,
or equivalently, if $x\approx_{X}y$ implies $f\left(x\right)\approx_{Y}f\left(y\right)$
for all $y\in\ns{X}$.
\end{defn}

It is clear that the collection of standard uniform pairs and internal
S-continuous maps forms a category $\Unif_{2,S}$. The category $\Unif_{2}$
can be regarded as a subcategory of $\Unif_{2,S}$ under the identification
of standard uniformly continuous maps $f\colon\left(X,A\right)\to\left(Y,B\right)$
with their nonstandard extensions $\ns{f}\colon\ns{\left(X,A\right)}\to\ns{\left(Y,B\right)}$.
Consider the singular homotopy equivalence within $\Unif_{2,S}$ defined
as follows.
\begin{defn}
Let $\left(X,A\right)$ and $\left(Y,B\right)$ be standard uniform
pairs. We say that two internal S-continuous maps $f,g\colon\ns{\left(X,A\right)}\to\ns{\left(Y,B\right)}$
are \emph{S-homotopic} if there exists an internal S-continuous homotopy
$h\colon\ns{\left(X,A\right)}\times\ns{\left[0,1\right]}\to\ns{\left(Y,B\right)}$
between $f$ and $g$.
\end{defn}

The S-continuous maps are precisely the maps that send microsimplices
to microsimplices. Hence every internal S-continuous map $f\colon\ns{\left(X,A\right)}\to\ns{\left(Y,B\right)}$
functorially induces homomorphisms $C_{\bullet}^{\mu}\left(f;G\right)\colon C_{\bullet}^{\mu}\left(X,A;G\right)\to C_{\bullet}^{\mu}\left(Y,B;G\right)$
and hence $H_{\bullet}^{\mu}\left(f;G\right)\colon H_{\bullet}^{\mu}\left(X,A;G\right)\to H_{\bullet}^{\mu}\left(Y,B;G\right)$.
In this setting, $\mu$-homology satisfies the S-homotopy axiom.
\begin{thm}[{\cite[Theorem 11]{Ima16}}]
\label{thm:S-homotopy-axiom-of-mu-hom}Let $G$ be an internal abelian
group. Let $\left(X,A\right)$ and $\left(Y,B\right)$ be standard
uniform pairs. If two internal S-continuous maps $f,g\colon\ns{\left(X,A\right)}\to\ns{\left(Y,B\right)}$
are S-homotopic, then $C_{\bullet}^{\mu}\left(f;G\right)$ and $C_{\bullet}^{\mu}\left(g;G\right)$
are chain-homotopic. Hence $H_{\bullet}^{\mu}\left(f;G\right)=H_{\bullet}^{\mu}\left(g;G\right)$.
\end{thm}

Next, we consider the relationship between uniform Vietoris homology
and S-homotopy equivalence.
\begin{defn}
Let $X$ and $Y$ be uniform spaces. Let $V$ be an entourage of $Y$.
We say that a map $f\colon X\to Y$ is\emph{ $V$\textendash continuous}
at $x\in X$ if there is an entourage $U$ of $X$ such that $f\left(U\left[x\right]\right)\subseteq V\left[f\left(x\right)\right]$.
We say that $f$ is \emph{uniformly $V$\textendash continuous} if
there is an entourage $U$ of $X$ such that $f\left(U\left[x\right]\right)\subseteq V\left[f\left(x\right)\right]$
for all $x\in X$.
\end{defn}

\begin{lem}
\label{lem:preshadow-lemma}Let $\left(X,A\right)$ and $\left(Y,B\right)$
be standard precompact uniform pairs. Let $f\colon\ns{\left(X,A\right)}\to\ns{\left(Y,B\right)}$
be an internal S-continuous map. For each entourage $V$ of $Y$,
there exists a (standard) uniformly $V$\textendash continuous map
$f^{V}\colon\left(X,A\right)\to\left(Y,B\right)$ (called a $V$\textendash preshadow
of $f$) such that $\ns{f^{V}}\left(x\right)\mathrel{\ns{V}}f\left(x\right)$
holds for all $x\in\ns{X}$.
\end{lem}

\begin{proof}
Let $V\in\mathcal{E}_{Y}$. Fix a symmetric $W\in\mathcal{E}_{Y}$
such that $W^{5}\subseteq V$, where $W^{n}$ refers to the $n$-fold
composition $W\circ W\circ\cdots\circ W$. By the nonstandard characterisation
of precompactness \cite[Theorem 8.4.34]{SL76}, for each $x\in X$,
we can choose a $y\in Y$ such that $y\mathrel{\ns{W}}f\left(x\right)$.
Similarly, for each $a\in A$, we can find a $b\in B$ such that $b\mathrel{\ns{W}}f\left(a\right)$.
Thus, we can define a (standard) map $f^{V}\colon\left(X,A\right)\to\left(Y,B\right)$
that satisfies $f^{V}\left(x\right)\mathrel{\ns{W}}f\left(x\right)$
on $X$.

According to the equivalent condition for S-continuity \cite[Theorem 8.4.23]{SL76},
there exists an entourage $U$ of $X$ such that $f\left(\ns{U}\left[x\right]\right)\subseteq\ns{W}\left[f\left(x\right)\right]$
for all $x\in X$. Let $x\in X$ and $y\in U\left[x\right]$. Then,
since $f^{V}\left(x\right)\mathrel{\ns{W}}f\left(x\right)\mathrel{\ns{W}}f\left(y\right)\mathrel{\ns{W}}f^{V}\left(y\right)$,
it follows that $f^{V}\left(x\right)\mathrel{\ns{V}}f^{V}\left(y\right)$.
By transfer, we have $f^{V}\left(x\right)\mathrel{V}f^{V}\left(y\right)$.
Hence $f^{V}\left(U\left[x\right]\right)\subseteq V\left[f^{V}\left(x\right)\right]$
for all $x\in X$. Therefore, $f^{V}$ is uniformly $V$\textendash continuous.

Let $x\in\ns{X}$. By \cite[Theorem 8.4.34]{SL76}, there exists a
$\xi\in X$ such that $\xi\mathrel{\ns{U}}x$. From the previous paragraph,
we have that $\ns{f^{V}}\left(x\right)\mathrel{\ns{W^{3}}}f^{V}\left(\xi\right)$
by transfer. Since $f^{V}\left(\xi\right)\mathrel{\ns{W}}f\left(\xi\right)\mathrel{\ns{W}}f\left(x\right)$,
we conclude that $\ns{f^{V}}\left(x\right)\mathrel{\ns{V}}f\left(x\right)$.
\end{proof}
\begin{rem}
This lemma can be easily refined as follows: let $\left(X,A\right)$
and $\left(Y,B\right)$ be standard uniform pairs. Let $f\colon\ns{\left(X,A\right)}\to\ns{\left(Y,B\right)}$
be an internal map such that $f\left(\ns{X}\right)\subseteq\mathrm{pns}\left(Y\right)$
and $f\left(\ns{A}\right)\subseteq\mathrm{pns}\left(B\right)$, where
$\mathrm{pns}\left(-\right)$ denotes the set of all prenearstandard
points. If $f$ is S-continuous on $X$, then for each entourage $V$
of $Y$, there exists a (non-uniformly) $V$\textendash continuous
map $f^{V}\colon\left(X,A\right)\to\left(Y,B\right)$ such that $f^{V}\left(x\right)\mathrel{\ns{V}}f\left(x\right)$
holds for all $x\in X$. Moreover, if $f$ is S-continuous also on
$\ns{X}\setminus X$, $f^{V}$ can be uniformly $V$\textendash continuous.
Furthermore, if $X$ is precompact, we can ensure that $\ns{f^{V}}\left(x\right)\mathrel{\ns{V}}f\left(x\right)$
holds also for all $x\in\ns{X}\setminus X$.
\end{rem}

\begin{thm}
Let $G$ be a standard abelian group. Let $\left(X,A\right)$ and
$\left(Y,B\right)$ be standard precompact uniform pairs. Every internal
S-continuous map $f\colon\ns{\left(X,A\right)}\to\ns{\left(Y,B\right)}$
functorially induces a homomorphism $\check{H}_{\bullet}^{u}\left(f;G\right)\colon\check{H}_{\bullet}^{u}\left(X,A;G\right)\to\check{H}_{\bullet}^{u}\left(Y,B;G\right)$.
\end{thm}

\begin{proof}
For each entourage $V$, fix a symmetric entourage $\sqrt{V}$ with
$\sqrt{V}^{2}\subseteq V$. By \prettyref{lem:preshadow-lemma}, for
each entourage $V$ of $Y$, there exists a $\sqrt{V}$\textendash preshadow
$f^{\sqrt{V}}\colon\left(X,A\right)\to\left(Y,B\right)$. For a sufficiently
small entourage $U$ of $X$, $f^{\sqrt{V}}$ is a simplicial map
from $\left(X_{U},A_{U}\right)$ to $\left(Y_{V},B_{V}\right)$. It
induces a homomorphism $f_{V}:=H_{\bullet}\left(f^{\sqrt{V}};G\right)\circ p_{U}\colon\check{H}_{\bullet}^{u}\left(X,A;G\right)\to H_{\bullet}\left(Y_{V},B_{V};G\right)$.

On the other hand, by \cite[Theorem 8.4.23]{SL76}, there exists a
symmetric entourage $U$ of $X$ such that $f\left(\ns{U}\left[x\right]\right)\subseteq\ns{\sqrt{V}}\left[f\left(x\right)\right]$
for all $x\in\prescript{\ast}{}{X}$. For simplicity, we abbreviate
$f^{\sqrt{V}}=\ns{f^{\sqrt{V}}}$. Taking $U$ to be sufficiently
small, we may assume that $f^{\sqrt{V}}\left(\ns{U}\left[x\right]\right)\subseteq\ns{\sqrt{V}}\left[f^{\sqrt{V}}\left(x\right)\right]$
holds for all $x\in\prescript{\ast}{}{X}$. Let us prove that $f$
and $f^{\sqrt{V}}$ are internally contiguous as internal simplicial
maps from $\ns{\left(X_{U},A_{U}\right)}$ to $\ns{\left(Y_{V},B_{V}\right)}$.
Let $\left(a_{0},\ldots,a_{p}\right)$ be an internal $p$\textendash simplex
of $\ns{X_{U}}$ (resp. $\ns{A_{U}}$). There exists an $x\in\ns{X}$
such that $a_{k}\mathrel{\ns{U}}x$ for all $0\leq k\leq p$. Since
$f\left(a_{k}\right)\mathrel{\ns{\sqrt{V}}}f\left(x\right)\mathrel{\ns{\sqrt{V}}}f^{\sqrt{V}}\left(x\right)$,
we have $f\left(a_{k}\right)\mathrel{\ns{V}}f^{\sqrt{V}}\left(x\right)$.
Moreover $f^{\sqrt{V}}\left(a_{k}\right)\mathrel{\ns{V}}f^{\sqrt{V}}\left(x\right)$
holds. Hence $\left(f\left(a_{0}\right),\ldots,f\left(a_{p}\right),f^{\sqrt{V}}\left(a_{0}\right),\ldots,f^{\sqrt{V}}\left(a_{p}\right)\right)$
is an internal $\left(2p+1\right)$\textendash simplex of $\ns{Y_{V}}$
(resp. $\ns{B_{V}}$).

Now, we will prove that the following diagram commutes for all $V\supseteq W$:
\[
\xymatrix{ & \check{H}_{\bullet}^{u}\left(X,A;G\right)\ar[dl]_{f_{V}}\ar[dr]^{f_{W}}\\
H_{\bullet}\left(Y_{V},B_{V};G\right)\ar[rr]^{p_{WV}} &  & H_{\bullet}\left(Y_{W},B_{W};G\right)
}
\]
As we proved above, $f$ and $f^{\sqrt{V}}$ are internally contiguous,
and so are $f$ and $f^{\sqrt{W}}$. It follows that $\ns{\left(H_{\bullet}\left(f^{\sqrt{W}};G\right)\right)}=\left(\ns{H_{\bullet}}\right)\left(f;\ns{G}\right)=\ns{\left(p_{WV}\circ H_{\bullet}\left(f^{\sqrt{V}};G\right)\right)}$.
By transfer, we have that $H_{\bullet}\left(f^{\sqrt{W}};G\right)=p_{WV}\circ H_{\bullet}\left(f^{\sqrt{V}};G\right)$,
and hence $f_{W}=p_{WV}\circ f_{V}$. By the universal property of
$\check{H}_{\bullet}^{u}\left(Y,B;G\right)$, we obtain a homomorphism
$\check{H}_{\bullet}^{u}\left(f;G\right)\colon\check{H}_{\bullet}^{u}\left(X,A;G\right)\to\check{H}_{\bullet}^{u}\left(Y,B;G\right)$.
We can also show that $\check{H}_{\bullet}^{u}\left(f;G\right)$ is
independent of the choice of preshadows.

Suppose that $f\colon\ns{\left(X,A\right)}\to\ns{\left(Y,B\right)}$
and $g\colon\ns{\left(Y,B\right)}\to\ns{\left(Z,C\right)}$ are internal
S-continuous maps, where $\left(X,A\right)$, $\left(Y,B\right)$
and $\left(Z,C\right)$ are standard precompact uniform pairs. Let
$V$ be an entourage of $Z$. Let $g^{\sqrt{V}}$ be a $\sqrt{V}$\textendash preshadow
of $g$. There is an entourage $U$ of $Y$ such that $g^{\sqrt{V}}\left(U\left[x\right]\right)\subseteq\sqrt{V}\left[g^{\sqrt{V}}\left(x\right)\right]$.
Let $f^{\sqrt{U}}$ be a $\sqrt{U}$-preshadow of $f$. Then, $g^{\sqrt{V}}\circ f^{\sqrt{U}}$
is a $\sqrt{V}$\textendash preshadow of $g\circ f$. There is an
entourage $W$ of $X$ such that $f^{\sqrt{U}}\left(W\left[x\right]\right)\subseteq\sqrt{U}\left[f^{\sqrt{U}}\left(x\right)\right]$
and $g^{\sqrt{V}}\circ f^{\sqrt{U}}\left(W\left[x\right]\right)\subseteq\sqrt{V}\left[g^{\sqrt{V}}\circ f^{\sqrt{U}}\left(x\right)\right]$
for all $x\in X$. The following diagram is commutative:
\[
\xymatrix{\check{H}_{\bullet}^{u}\left(X,A;G\right)\ar[d]\ar[r]^{\check{H}_{\bullet}^{u}\left(f;G\right)}\ar@(u,u)[rr]^{\check{H}_{\bullet}^{u}\left(g;G\right)\circ\check{H}_{\bullet}^{u}\left(f;G\right)} & \check{H}_{\bullet}^{u}\left(Y,B;G\right)\ar[d]\ar[r]^{\check{H}_{\bullet}^{u}\left(g;G\right)} & \check{H}_{\bullet}^{u}\left(Z,C;G\right)\ar[d]\\
H_{\bullet}\left(X_{W},A_{W};G\right)\ar[r]^{H_{\bullet}\left(f^{\sqrt{U}};G\right)}\ar@(d,d)[rr]_{H_{\bullet}\left(g^{\sqrt{V}}\circ f^{\sqrt{U}};G\right)} & H_{\bullet}\left(Y_{U},B_{U};G\right)\ar[r]^{H_{\bullet}\left(g^{\sqrt{V}};G\right)} & H_{\bullet}\left(Z_{V},C_{V};G\right)
}
\]
We have that $\check{H}_{\bullet}^{u}\left(g\circ f;G\right)=\check{H}_{\bullet}^{u}\left(g;G\right)\circ\check{H}_{\bullet}^{u}\left(f;G\right)$.
\end{proof}
\begin{thm}
\label{thm:S-homotopy-axiom-of-uV-hom}Let $G$ be a standard abelian
group. Let $\left(X,A\right)$ and $\left(Y,B\right)$ be standard
precompact uniform pairs. If two internal S-continuous maps $f,g\colon\ns{\left(X,A\right)}\to\ns{\left(Y,B\right)}$
are S-homotopic, then $\check{H}_{\bullet}^{u}\left(f;G\right)=\check{H}_{\bullet}^{u}\left(g;G\right)$.
\end{thm}

\begin{proof}
Let $h\colon\ns{\left(X,A\right)}\times\ns{\left[0,1\right]}\to\ns{\left(Y,B\right)}$
be an internal S-homotopy between $f$ and $g$. Fix an infinite $N\in\ns{\mathbb{N}}$.
Define $h_{i}:=h\left(\cdot,i/N\right)$ for $i=0,1,\ldots,N$. Let
$V$ be an entourage of $Y$. Let $\sqrt{V}$ be a symmetric entourage
of $Y$ with $\sqrt{V}^{2}\subseteq V$. By \cite[Theorem 8.4.23]{SL76},
there exists a symmetric entourage $U$ of $X$ such that $h_{i}\left(\ns{U}\left[x\right]\right)\subseteq\ns{\sqrt{V}}\left[h_{i}\left(x\right)\right]$
for all $x\in\ns{X}$ and $0\leq i\leq N$.

All $h_{i}$s are internal simplicial maps from $\ns{\left(X_{U},A_{U}\right)}$
to $\ns{\left(Y_{V},B_{V}\right)}$. Let us prove that $h_{i}$ and
$h_{i+1}$ are internally contiguous for all $0\leq i<N$. Let $\left(a_{0},\ldots,a_{p}\right)$
be a $p$\textendash simplex of $\ns{X_{U}}$ (resp. $\ns{A_{U}}$).
There exists an $x\in\ns{X}$ with $\left(a_{k},x\right)\in\ns{U}$
for all $0\leq k\leq n$. Since $h_{i}\left(a_{k}\right)\mathrel{\ns{\sqrt{V}}}h_{i}\left(x\right)\mathrel{\ns{\sqrt{V}}}h_{i+1}\left(x\right)$,
we have $h_{i}\left(a_{k}\right)\mathrel{\ns{V}}h_{i+1}\left(x\right)$.
Moreover $h_{i+1}\left(a_{k}\right)\mathrel{\ns{V}}h_{i+1}\left(x\right)$
holds. Hence $\left(h_{i}\left(a_{0}\right),\ldots,h_{i}\left(a_{p}\right),h_{i+1}\left(a_{0}\right),\ldots,h_{i+1}\left(a_{p}\right)\right)$
is a $\left(2p+1\right)$\textendash simplex of $\ns{Y_{V}}$ (resp.
$\ns{B_{V}}$). It follows that $\left(\ns{H_{\bullet}}\right)\left(f;\ns{G}\right)=\left(\ns{H_{\bullet}}\right)\left(g;\ns{G}\right)$
and $\ns{f_{V}}=\ns{g_{V}}$. By transfer, we have $f_{V}=g_{V}$
and therefore $\check{H}_{\bullet}^{u}\left(f;G\right)=\check{H}_{\bullet}^{u}\left(g;G\right)$.
\end{proof}
\begin{thm}[{\cite[Theorem 12]{Ima16}}]
Let $\left(X,A\right)$ be a standard uniform pair and let $\left(Y,B\right)$
be a subpair of $\left(X,A\right)$. Suppose that $Y$ and $B$ are
dense in $X$ and $A$, respectively. Then there exists an internal
S-deformation retraction $r\colon\ns{\left(X,A\right)}\to\ns{\left(Y,B\right)}$.
\end{thm}

Combining with the S-homotopy axiom (\prettyref{thm:S-homotopy-axiom-of-mu-hom}
and \prettyref{thm:S-homotopy-axiom-of-uV-hom}), we obtain the following
two results.
\begin{cor}
\label{cor:dense-subset-has-the-same-mu-hom-of-superset}Let $G$
be an internal abelian group. Let $\left(X,A\right)$ be a standard
uniform pair and let $\left(Y,B\right)$ be a subpair of $\left(X,A\right)$.
Suppose that $Y$ and $B$ are dense in $X$ and $A$, respectively.
Then the inclusion map $i\colon\left(Y,B\right)\hookrightarrow\left(X,A\right)$
induces an isomorphism $H_{\bullet}^{\mu}\left(i;G\right)\colon H_{\bullet}^{\mu}\left(Y,B;G\right)\cong H_{\bullet}^{\mu}\left(X,A;G\right)$.
\end{cor}

\begin{cor}
\label{cor:dense-subset-has-the-same-uVhom-of-superset}Let $G$ be
an abelian group. Let $\left(X,A\right)$ be a precompact uniform
pair and let $\left(Y,B\right)$ be a subpair of $\left(X,A\right)$.
Suppose that $Y$ and $B$ are dense in $X$ and $A$, respectively.
Then the inclusion map $i\colon\left(Y,B\right)\hookrightarrow\left(X,A\right)$
induces an isomorphism $\check{H}_{\bullet}^{u}\left(i;G\right)\colon\check{H}_{\bullet}^{u}\left(Y,B;G\right)\cong\check{H}_{\bullet}^{u}\left(X,A;G\right)$.
\end{cor}

Notice that the latter generalises the fact that $\mathbb{Q}/\mathbb{Z}$
and $\mathbb{R}/\mathbb{Z}$ have the same uniform Vietoris homology.

\subsection{The precompact case}

Based on these preparations, we now prove the main results in this
section.
\begin{thm}
Let $G$ be a standard abelian group. Let $\left(X,A\right)$ be a
standard precompact uniform pair. Then, $H_{\bullet}^{\mu}\left(X,A;\ns{G}\right)\cong\check{H}_{\bullet}^{u}\left(X,A;\ns{G}\right)$.
\end{thm}

\begin{proof}
Let $\bar{X}$ be a (standard) uniform completion of $X$. Let $\bar{A}$
be the closure of $A$ in $\bar{X}$. Since $X$ is precompact, $\bar{X}$
is compact, and so is $\bar{A}$. By \prettyref{cor:dense-subset-has-the-same-mu-hom-of-superset},
$H_{\bullet}^{\mu}\left(X,A;\ns{G}\right)\cong H_{\bullet}^{\mu}\left(\bar{X},\bar{A};\ns{G}\right)$.
By \prettyref{prop:The-compact-case}, $H_{\bullet}^{\mu}\left(\bar{X},\bar{A};\ns{G}\right)\cong\check{H}_{\bullet}^{u}\left(\bar{X},\bar{A};\ns{G}\right)$.
Finally, by \prettyref{cor:dense-subset-has-the-same-uVhom-of-superset},
$\check{H}_{\bullet}^{u}\left(\bar{X},\bar{A};\ns{G}\right)\cong\check{H}_{\bullet}^{u}\left(X,A;\ns{G}\right)$.
The proof is completed.
\end{proof}
\begin{cor}
Let $G$ be a standard abelian group. Let $\left(X,A\right)$ be a
standard pseudocompact completely regular pair such that $A$ is normally
embedded in $X$. Then, $H_{\bullet}^{K}\left(X,A;\ns{G}\right)\cong\check{H}_{\bullet}^{n}\left(X,A;\ns{G}\right)$.
\end{cor}

\begin{proof}
By \prettyref{prop:K=00003DMuF}, we have that $H^{K}\left(X,A;\ns{G}\right)=H_{\bullet}^{\mu}\left(FX,A;\ns{G}\right)$.
Since $X$ is pseudocompact, $FX$ is precompact. Hence $H_{\bullet}^{\mu}\left(FX,A;\ns{G}\right)\cong\check{H}_{\bullet}^{u}\left(FX,A;\ns{G}\right)$.
By \cite[Theorem 20]{Isb64}, we have that $\check{H}_{\bullet}^{u}\left(FX,A;\ns{G}\right)\cong\check{H}_{\bullet}^{n}\left(X,A;\ns{G}\right)$.
\end{proof}

\subsection{Counterexample in the general case}

The above-mentioned isomorphisms do not hold in general. The cause
is that the elementary embedding $\mathbb{U}\hookrightarrow\ns{\mathbb{U}}$
does not preserve infinitary operations such as infinite direct sums.
\begin{prop}
\label{prop:non-equivalence-theorem}There exist a standard uniform
space $X$ and a standard abelian group $G$ such that $H_{\bullet}^{\mu}\left(X;\ns{G}\right)\not\cong\check{H}_{\bullet}^{u}\left(X;\ns{G}\right)$.
\end{prop}

\begin{proof}
Let $X$ be the discrete uniform space $\mathbb{N}$. Let $G$ be
any (standard) nontrivial finite abelian group. Then, by transfer,
$\ns{G}=G$. It is easy to see that $\check{H}_{0}^{u}\left(X;G\right)\cong G^{\oplus\mathbb{N}}$.
The cardinality of $\check{H}_{0}^{u}\left(X;G\right)$ is $\aleph_{0}$.
On the other hand, we have $H_{0}^{\mu}\left(X;G\right)\cong\ns{\left(G^{\oplus\mathbb{N}}\right)}$.
By weak saturation, each element of $G^{\mathbb{N}}$ can be extended
to an element of $\ns{\left(G^{\oplus\mathbb{N}}\right)}$. Since
the cardinality of $G^{\mathbb{N}}$ is $2^{\aleph_{0}}$, the cardinality
of $H_{0}^{\mu}\left(X;G\right)$ must be at least $2^{\aleph_{0}}$.
Hence $\check{H}_{0}^{u}\left(X;G\right)$ and $H_{0}^{\mu}\left(X;G\right)$
cannot be isomorphic.
\end{proof}
\begin{cor}
There exist a standard completely regular space $X$ and a standard
abelian group $G$ such that $H_{\bullet}^{K}\left(X;\ns{G}\right)\not\cong\check{H}_{\bullet}^{n}\left(X;\ns{G}\right)$.
\end{cor}

\begin{proof}
Let $X$ and $G$ be the same as in \prettyref{prop:non-equivalence-theorem}.
Since $X$ is fine, we have that $H_{\bullet}^{K}\left(UX;G\right)=H_{\bullet}^{\mu}\left(X;G\right)\not\cong\check{H}_{\bullet}^{u}\left(X;G\right)=\check{H}_{\bullet}^{n}\left(UX;G\right)$.
\end{proof}

\section{\label{sec:Homology-of-nonstandard-subsets}Homology of nonstandard
subsets}

Let $X$ be a standard uniform space. We call a (possibly external)
subset $A$ of $\ns{X}$ a \emph{uniform $\mathcal{E}$-set}, and
denote the whole space $X$ by $\mathrm{Amb}\left(A\right)$. The
notion of uniform $\mathcal{E}$-set is the uniform analogue of $\mathcal{E}$-set
(see Wattenberg \cite{Wat78}). We say that a map $f\colon A\to B$
between uniform $\mathcal{E}$-sets is an \emph{S-map} if $f$ has
an internal extension $F\colon\ns{\mathrm{Amb}\left(A\right)}\to\ns{\mathrm{Amb}\left(B\right)}$
that is S-continuous on $A$. The collection of uniform $\mathcal{E}$-sets
and S-maps forms a category $\uE_{S}$. We denote by $\uE_{2,S}$
the category of pairs of uniform $\mathcal{E}$-sets with S-maps.

$\mu$-homology can be extended to $\uE_{S}$ (and to $\uE_{2,S}$)
as follows. First, define the \emph{microchain complex} of $A\in\uE_{S}$
by
\[
C_{p}^{\mu}\left(A;G\right):=\Set{\sum_{i}g_{i}\sigma_{i}\in C_{p}^{\mu}\left(\mathrm{Amb}\left(A\right);G\right)|\sigma_{i}\subseteq A\text{ for all }i},
\]
where $C_{p}^{\mu}\left(\mathrm{Amb}\left(A\right)\right)$ in the
right hand side is defined as in \prettyref{sec:Preliminaries}. Given
an S-map $f\colon A\to B$, choose an internal extension $F\colon\ns{\mathrm{Amb}\left(A\right)}\to\ns{\mathrm{Amb}\left(B\right)}$
that is S-continuous on $A$, and define a homomorphism $C_{\bullet}^{\mu}\left(f;G\right)\colon C_{\bullet}^{\mu}\left(A;G\right)\to C_{\bullet}^{\mu}\left(B;G\right)$
by letting
\[
C_{p}^{\mu}\left(f;G\right)\left(a_{0},\ldots,a_{p}\right):=\left(\ns{F}\left(a_{0}\right),\ldots,\ns{F}\left(a_{p}\right)\right).
\]
It is independent of the choice of $F$. Obviously $C_{\bullet}^{\mu}$
is a functor on $\uE_{S}$. Finally, let $H_{\bullet}^{\mu}\left(\cdot;G\right):=H_{\bullet}C_{\bullet}^{\mu}\left(\cdot;G\right)$.

We can show (in the same way as in \cite{Ima16}) that the extended
$\mu$-homology satisfies the homotopy, exactness, excision, dimension
and finite additivity axioms. The excision axiom is formulated as
follows.
\begin{defn}
Let $X$ be a uniform $\mathcal{E}$-set. Let $A$ and $B$ be subsets
of $X$. We say that $A$ is \emph{strongly contained in $B$} (and
we write $A\Subset B$) if $\mu_{\mathrm{Amb}\left(X\right)}^{u}\left(A\right)\cap X\subseteq B$.
\end{defn}

\begin{prop}
Let $X$ be a uniform $\mathcal{E}$-set. Let $A,B$ be subsets of
$X$ such that either of them is internal. If $X\setminus A\Subset B$
(or $X\setminus B\Subset A$), then the inclusion map $i\colon\left(A,A\cap B\right)\hookrightarrow\left(X,B\right)$
induces an isomorphism $H_{\bullet}^{\mu}\left(i;G\right)\colon H_{\bullet}^{\mu}\left(A,A\cap B;G\right)\cong H_{\bullet}^{\mu}\left(X,B;G\right)$.
\end{prop}

\begin{proof}
Similarly to \prettyref{prop:Excision-of-mu-hom}, we can prove that
each microsimplex of $X$ is contained in either $A$ or $B$. Given
$u:=\sum_{i}g_{i}\sigma_{i}\in C_{p}^{\mu}\left(X;G\right)$, if $A$
is internal, then $u$ can be decomposed into $\sum_{\sigma_{i}\subseteq A}g_{i}\sigma_{i}\in C_{p}^{\mu}\left(A;G\right)$
and $\sum_{\sigma_{i}\nsubseteq A}g_{i}\sigma_{i}\in C_{p}^{\mu}\left(B;G\right)$.
If $B$ is internal, then $u$ can be decomposed into $\sum_{\sigma_{i}\subseteq B}g_{i}\sigma_{i}\in C_{p}^{\mu}\left(B;G\right)$
and $\sum_{\sigma_{i}\nsubseteq B}g_{i}\sigma_{i}\in C_{p}^{\mu}\left(A;G\right)$.
\end{proof}
\begin{rem}
The internality condition of the excision axiom cannot be omitted.
Otherwise, the excision axiom causes a contradiction as follows. Let
$G$ be any standard nontrivial finite abelian group. (Note that $G=\ns{G}$
by transfer.) By applying the Mayer-Vietoris theorem to the triad
$\left(\ns{\mathbb{R}};\mu_{\mathbb{R}}\left(0\right),\ns{\mathbb{R}}\setminus\mu_{\mathbb{R}}\left(0\right)\right)$,
we obtain the following exact sequence:
\[
\xymatrix{0\ar[r] & H_{0}^{\mu}\left(\mu_{\mathbb{R}}\left(0\right);G\right)\oplus H_{0}^{\mu}\left(\ns{\mathbb{R}}\setminus\mu_{\mathbb{R}}\left(0\right);G\right)\ar[r] & H_{0}^{\mu}\left(\ns{\mathbb{R}};G\right)\ar[r] & 0}
\]
It is not difficult to see that $H_{0}^{\mu}\left(\mu_{\mathbb{R}}\left(0\right);G\right)\cong G$,
$H_{0}^{\mu}\left(\ns{\mathbb{R}}\setminus\mu_{\mathbb{R}}\left(0\right);G\right)\cong G\oplus G$
and $H_{0}^{\mu}\left(\ns{\mathbb{R}};G\right)\cong G$. However,
by exactness $G\oplus G\oplus G\cong H_{0}^{\mu}\left(\mu_{\mathbb{R}}\left(0\right);G\right)\oplus H_{0}^{\mu}\left(\ns{\mathbb{R}}\setminus\mu_{\mathbb{R}}\left(0\right);G\right)\cong H_{0}^{\mu}\left(\ns{\mathbb{R}};G\right)\cong G$,
which makes a contradiction.
\end{rem}

\section{\label{sec:Use-of-the-saturation-principle}Use of the saturation
principle}

The saturation principle is not necessary to construct the microsimplicial
homology theories. The satisfaction of the Eilenberg-Steenrod axioms
can be proved without the \emph{full} saturation principle (only using
the \emph{weak} saturation principle, i.e. the enlargement property
of the nonstandard universe). The equivalences among the microsimplicial
homology theories can also be proved without saturation. On the other
hand, the equivalences between the Vietoris-type and microsimplicial
homology theories depend on saturation (see \cite[Theorem 4]{Gar78}).
The proofs of the continuities of Korppi and $\mu$-homology also
depend on saturation.

\bibliographystyle{amsalpha}
\bibliography{\string"Relationship among various Vietoris-type and microsimplicial homology theories\string"}

\end{document}